\documentclass[12pt]{amsart}
 
\usepackage[dvips]{graphicx}
	\usepackage{amsmath,amssymb,amsthm,wrapfig,amsfonts,enumerate,latexsym}

	\newtheorem{dfn}{Definition}[section]
	\newtheorem{thm}[dfn]{Theorem}

	\newtheorem{lem}[dfn]{Lemma}
	\newtheorem{rem}[dfn]{Remark}
	
 	\newtheorem{claim}[dfn]{Claim}

	\newtheorem{ack}{Acknowledgements\!\!}

	\newcounter{yon}
	\numberwithin{equation}{section}

	\newcommand{\dist}{\mathop{\mathit{d}} \nolimits}
	
	\newcommand{\diam}{\mathop{\mathrm{diam}} \nolimits}

	\newcommand{\sep}{\mathop{\mathrm{Sep}} \nolimits}
	
	\newcommand{\obs}{\mathop{\mathrm{ObsDiam}}  \nolimits}

	\newcommand{\supp}{\mathop{\mathrm{Supp}}    \nolimits}

	\newcommand{\obsc}{\mathop{\mathrm{ObsCRad}}               \nolimits}
	\newcommand{\crad}{\mathop{\mathrm{CRad}}           \nolimits}

    \newcommand{\wa}{\mathop{{\sum\limits_{i=1,\cdots,n}^{\longrightarrow}}}
    \nolimits}

\begin{document}

	\title[]{Rate of convergence of
	stochastic processes with values in
	$\mathbb{R}$-trees and Hadamard manifolds}
	\author[Kei Funano]{Kei Funano}
	\address{Department of Mathematics and Engineering, Graduate School
    of Science and Technology, Kumamoto university, Kumamoto 860-8500, JAPAN}
	\email{yahoonitaikou@gmail.com}
	\thanks{}
	\dedicatory{}
	\date{\today}

	\maketitle

%\tableofcontents

\begin{abstract}Under K.-T. Sturm's formulation, we obtain a Gaussian
 upper bound for tail probability of mean value of independent, identically
 distributed random variables with values in $\mathbb{R}$-trees and
 Hadamard manifolds.
 \end{abstract}
	\setlength{\baselineskip}{5mm}

%    \tableofcontents
\section{Introduction and statement of the main result}The aim of this
paper is to study the weak Law of Large numbers for
CAT$(0)$-space-valued stochastic processes (see Subsection 2.1 for the
definition of CAT$(0)$-spaces).

Let $N$ be a CAT$(0)$-space and $(\Omega ,\Sigma, \mathbb{P})$ a
probability space. Given a random variable $W:\Omega \to N$
such that the push-forward measure $W_{\ast}\mathbb{P}$ of $\mathbb{P}$ by $W$ has the finite
moment of order $2$, we define its \emph{expectation} $\mathbb{E}_{\mathbb{P}}(W)$ by
the barycenter of the measure $W_{\ast}\mathbb{P}$ (the definition of
the barycenter is in Subsection 2.1). In \cite[Theorem 4.7]{sturm},
K.-T. Sturm introduced a natural definition of mean value of $n$-points
$y_1,\cdots,y_n$ in $N$, called \emph{inductive mean value} and denoted
by $\frac{1}{n}\wa y_i$ (see Definition \ref{cadef1} for precise definition). For an
independent, identically distributed $N$-valued random variables
$(Y_i)_{i=1}^{\infty}$ on the probability space $\Omega$, he obtained
the weak Law of Large numbers proving the following inequality
\begin{align}\label{ints1}
 \int_{\Omega} \dist_N\Big(\frac{1}{n}\wa Y_i(\omega),
 \mathbb{E}_{\mathbb{P}}(Y_1)\Big)^2 d \mathbb{P}(\omega)\leq
 \frac{1}{n}\int_{\Omega}\dist_N(Y_1(\omega),\mathbb{E}_{\mathbb{P}}(Y_1))^2 d\mathbb{P}(\omega).
 \end{align}He also proved the strong Law of Large numbers
 (\cite[Theorem 4.7, Proposition 6.6]{sturm}).

 Motivated by Sturm's work, using the results of the theory of
 L\'{e}vy-Milman concentration of $1$-Lipschitz maps obtained in
 \cite{funa2,funa3}, we obtain the following Gaussian estimate.

\begin{thm}\label{mth1}Let $(Y_i)_{i=1}^{\infty}$ be a sequence of independent,
 identically distributed random variables on a probability space
 $(\Omega,\Sigma, \mathbb{P})$ with values in an $\mathbb{R}$-tree
 $T$. We assume that the support of the measure
 $(Y_1)_{\ast}\mathbb{P}$ has bounded diameter $D$. Then, for any $r>0$,
 we have
 \begin{align*}
  \mathbb{P} \Big( \Big\{ \omega \in \Omega \mid \dist_T\Big( \frac{1}{n}\wa
  Y_i(\omega),\mathbb{E}_{\mathbb{P}}(Y_1)\Big)\geq r \Big\}\Big)\leq 4e^{\frac{4}{75}}
  e^{-\frac{nr^2}{150  D^2 }}.
  \end{align*}
 \end{thm}See Subsection 2.1 for definition of $\mathbb{R}$-trees.

 In the case where $N$ is an Hadamard manifold, we also obtain
  the following. For any $m\in \mathbb{N}$, we put
 \begin{align*}
  A_m:= e^{1/(2m)}\Big\{ 1+
  \frac{\sqrt{\pi}e^{(m+1)/(4m-2)}e^{\pi^2}}{2}\Big\} \text{ and }\widetilde{A}_m:=e^{1/(4m)}\{1+\sqrt{\pi}e^{(m+1)/(4m-2)}\}.
  \end{align*}Note that both $A_m$ and $\widetilde{A}_m$ are bounded
  from above by universal constant $C>0$. 
 \begin{thm}\label{mth2}Let $(Y_i)_{i=1}^{\infty}$ be a sequence of independent,
 identically distributed random variables on a probability space
 $(\Omega,\Sigma, \mathbb{P})$ with values in an $m$-dimensional Hadamard manifold
  $N$. We assume that the support of the measure
 $(Y_1)_{\ast}\mathbb{P}$ has bounded diameter $D$. Then, for any $r>0$,
 we have
  \begin{align*}
   \mathbb{P}\Big( \Big\{ \omega \in \Omega \mid \dist_N \Big( \frac{1}{n}\wa  Y_i(\omega)  , \mathbb{E}_{\mathbb{P}}(Y_1)
    \Big)\geq r \Big\}\Big)\leq \min \{       A_m e^{-\frac{nr^2}{16D^2m}},
   \widetilde{A}_m e^{-\frac{nr^2}{32D^2m}}\}
   \end{align*}
  \end{thm}

  There are many other way to define a mean value of points in a
  CAT$(0)$-space (see Remark \ref{car1}). For example, in \cite{sahib}, A. Es-Sahib and
  H. Heinich introduced an another notion of mean value and expectation.
  They obtained the strong Law of Large numbers under
  their definition. In this paper, we treat only Sturm's formulation.

  \begin{ack}\upshape The author would like to thank to Professors
   Kazuhiro Kuwae and Daehong Kim for motivating this work. This work was partially supported by Research Fellowships of
	the Japan Society for the Promotion of Science for Young Scientists.
   \end{ack}
\section{Preliminaries}

 \subsection{Basics of CAT(0)-spaces}
In this subsection we explain several terminologies in geometry of
CAT$(0)$-spaces. We refer to \cite{sturm} for the details of
the results on CAT$(0)$-spaces mentioned below.

Let $(X,\dist_X)$ be a metric space. A rectifiable curve
$\gamma:[0,1]\to X$ is called a \emph{geodesic} if its arclength
coincides with the distance $\dist_X(\gamma(0),\gamma(1))$ and it has a
constant speed, i.e., parameterized proportionally to the arclength. We
say that a metric space is a \emph{geodesic space}
if any two points are joined by a geodesic between them. If any
two points are joined by a unique geodesic, then the space is said to be \emph{uniquely geodesic}. A
complete geodesic space $X$ is called a \emph{CAT$(0)$-space} if we have
\begin{align*}
\dist_X(x,\gamma (1/2))^2 \leq \frac{1}{2}\dist_X(x,y)^2 +
 \frac{1}{2}\dist_X(x,z)^2 - \frac{1}{4} \dist_X(y,z)^2
\end{align*}for any $x,y,z \in X$ and any geodesic $\gamma
 :[0,1]\to X$ from $y$ to $z$. For example, Hadamard manifolds, Hilbert
 spaces, and $\mathbb{R}$-trees are all CAT(0)-spaces. An
 \emph{$\mathbb{R}$-tree} is a complete geodesic space such that the image of
 every simple path is the image of a geodesic.

 It follows from the next
 theorem that CAT$(0)$-spaces are uniquely geodesic.
\begin{thm}[{cf.~\cite[Corollary 2.5]{sturm}}]\label{cat1}Let $N$ be a
  CAT$(0)$-space and $\gamma,\eta:[0,1]\to N$ be two geodesics. Then,
  for any $t\in [0,1]$, we have
  \begin{align*}
   \dist_N(\gamma(t) , \eta(t))\leq (1-t)\dist_N(\gamma(0),\eta(0))+t\dist_N(\gamma(1),\eta(1))
   \end{align*}
  \end{thm}

 Let $N$ be a CAT$(0)$-space. We denote by $\mathcal{P}^2(N)$ the
 set of all Borel probability measure $\nu $ on $N$
 having the finite moment of order $2$, i.e.,
 \begin{align*}
  \int_N \dist_N(x,y)^2d\nu(y)<+\infty
  \end{align*}for some (hence all) $x\in N$. A point $x_0\in N$ is
  called the \emph{barycenter} of a measure $\nu\in \mathcal{P}^2(N)$ if $x_0$ is the
  unique minimizing point of the function
  \begin{align*}
   N \ni x\mapsto \int_N \dist_N(x,y)^2d\nu(y)\in \mathbb{R}.
   \end{align*}We denote the point $x_0$ by $b(\nu)$. It is well-known
   that every $\nu \in \mathcal{P}^2(N)$ has the barycenter
   (\cite[Proposition 4.3]{sturm}).

A simple variational argument implies the following lemma.
\begin{lem}[{cf.~\cite[Proposition 5.4]{sturm}}]\label{cal1}Let $H$ be a
 Hilbert space. Then, for each
 $\nu\in \mathcal{P}^2(H)$, we have 
 \begin{align*}
  b (\nu)= \int_{H} y  d\nu(y).
  \end{align*}
\end{lem}
% \begin{lem}[{cf.~\cite[Proposition 5.4]{sturm}}]\label{bl2}Let $N$ be an Hadamard
%  manifold and $\nu\in \mathcal{P}^2(N)$. Then $x= b(\nu)$ if and only if
%  \begin{align*}
%   \int_N \exp^{-1}_x (y) d\nu(y)=0.
%   \end{align*}In particular, identifying the tangent space of $N$ at
%  $b(\nu)$ with the Euclidean space of the same dimension, we have $b ((\exp^{-1}_{b(\nu)})_{\ast}(\nu))=0$.
%  \end{lem}

     Let $(\Omega, \Sigma, \mathbb{P})$ a probability space and $N$ a
     CAT($0$)-space. For an $N$-valued random variables $W:\Omega \to
     N$ satisfying $W_{\ast}\mathbb{P}\in
     \mathcal{P}^2(N)$, we define its \emph{expectation}
     $\mathbb{E}_{\mathbb{P}}(f)\in N$ by the point
$b(W_{\ast}\mathbb{P})$. By Lemma \ref{cal1}, in the case where $N$ is a
     Hilbert space, this definition
coincides with the classical one:
\begin{align*}
 \mathbb{E}_{\mathbb{P}}(W)=\int_{\Omega}W(\omega)d\mathbb{P}(\omega).
 \end{align*}

   The proof of the next lemma is easy, so we omit it.
  \begin{lem}\label{cal2}Let $N$ be a CAT$(0)$-space and $\nu \in
   \mathcal{P}^2(N)$. Then, we have
   \begin{align*}
    \dist_N(b(\nu), \supp \nu)\leq \diam (\supp \nu).
    \end{align*}
   \end{lem}
  \begin{thm}[{Variance inequality, cf.~\cite[Proposition
   4.4]{sturm}}]\label{cal3}Let $N$ be a CAT$(0)$-space and $\nu\in
   \mathcal{P}^2(N)$. Then, for any $z\in N$, we have 
   \begin{align*}
    \int_N \{\dist_N(z,x)^2- \dist_N(b(\nu),x)^2\} d\nu(x)\geq \dist_N(z,b(\nu))^2 
    \end{align*}
   \end{thm}

   We now explain the inductive mean value introduced by Sturm in \cite[Definition
    4.6]{sturm}.
   \begin{dfn}[{Inductive mean value}]\label{cadef1}\upshape Given a sequence $(y_i)_{i=1}^{\mathbb{N}}$ of points
    in a uniquely geodesic space $X$, we define a new sequence of points $s_n \in
    X$, $n\in \mathbb{N}$, by induction as follows. We define $s_1:=y_1$
    and $s_n:=\gamma (1/n)$, where $\gamma:[0,1]\to X$ is the geodesic 
    connecting two points $s_{n-1}$ and $y_n$. We denote the point $s_n$
    by $\frac{1}{n}\wa y_i$ and call it the \emph{inductive mean value}
    of the points $y_1,\cdots, y_n$.
    \end{dfn}

    \begin{rem}\label{car1}\upshape
      $(1)$ If the space $X$ is a non-linear metric space, then the
     point $\frac{1}{n}\wa y_i$ strongly depends on permutations of
     $y_i$ as we see the following example. For $i=1,2,3$, let $T_i:=\{ (i,r) \mid r\in [0,+\infty)\}$ be a copy of
   $[0,+\infty)$ equipped with the usual Euclidean distance
   function. The \emph{tripod} $T$ is the metric space obtained by gluing together all these
   spaces $T_i$, $i=1,2,3$, at their origins with the intrinsic distance
   function. Let $y_1:=(1,1)$, $y_2:=(2,1)$, and $y_3:=(3,1)$. Then, the inductive mean value of order $y_1,y_2,y_3$ is
     the point $(3,1/2)$, whereas the one of order $y_1,y_3,y_2$ is the
     point $(2,1/2)$. 

     $(2)$ There are many other way to define a mean value of points
     $y_1,\cdots , y_n$ in a CAT$(0)$-space (see \cite[Remark
     $6.4$]{sturm}). For example, define a
     mean value as the barycenter of these points. Observe that this
     definition does not depend on order of the points (and so it is different
     from inductive mean value in general).
    \end{rem}
 \subsection{Invariants of mm-spaces and measures}In this subsection we
     define several invariants of mm-spaces and measures, which are
     needed for the proof of the main theorems.

An \emph{mm-space} $X=(X,\dist_X,\mu_X)$ is a complete separable metric
space $(X,\dist_X)$ with a Borel probability measure $\mu_X$. Let $Y$ be a complete metric space and $\nu$ a finite Borel measure on
          $Y$ having separable support with the total measure $m$. For
          any $\kappa>0$, we define the \emph{partial diameter} $\diam(\nu,m-\kappa)$ of
          $\nu$ as the infimum of the diameter of $Y_0$, where
          $Y_0$ runs over all Borel subsets of $Y$ such that
          $\nu( Y_0)\geq m-\kappa$. Let $X$ be
      an mm-space with $m_X:=\mu_X(X)$ and $Y$ a complete metric space. For any $\kappa >0$, we
	 define the \emph{observable diameter} of $X$ by 
	 \begin{align*}
	  \obs_Y (X; -\kappa):=
	   \sup \{ \diam (f_{\ast}(\mu_X),m_X-\kappa) \mid f:X\to Y \text{ is a
      }1 \text{{\rm -Lipschitz map}}  \}. 
      \end{align*}The idea of the observable diameter comes from the quantum and statistical
	mechanics, i.e., we think of $\mu_X$ as a state on a configuration
	space $X$ and $f$ is interpreted as an observable.
 
Let $X$ be an mm-space. Given any two positive numbers $\kappa_1$ and $\kappa_2$, we define the
\emph{separation distance}
$\sep(X;\kappa_1,\kappa_2)=\sep(\mu_X;\kappa_1,\kappa_2)$ of $X$ as the
supremum of the number $\dist_X(A_1,A_2)$, where $A_1$ and $A_2$ are
Borel subsets of $X$ such that $\mu_X(A_1)\geq \kappa_1$ and
$\mu_X(A_2)\geq \kappa_2$, and we put
\begin{align*}
 \dist_X(A_1,A_2):=\inf \{ \dist_X(x_1,x_2) \mid x_1\in A_1 ,x_2\in A_2\}.
 \end{align*}

 The next two lemmas are easy to prove.
 \begin{lem}[{cf.~\cite[Section $3\frac{1}{2}.30$]{gromov}}]\label{inl1}Let $X$ and $Y$ be two mm-spaces and $f:X\to Y$ be a
  $\alpha$-Lipschitz map such that $f_{\ast}(\mu_X)=\mu_Y$. Then, for any
  $\kappa_1,\kappa_2>0$, we have
  \begin{align*}
   \sep (Y;\kappa_1,\kappa_2)\leq \alpha \sep(X;\kappa_1,\kappa_2).
   \end{align*}
  \end{lem}

  \begin{lem}\label{inl2}Given two positive numbers $\kappa_1$ and $\kappa_2$ such
   that $\kappa_1\geq 1/2$ and $\kappa_2>1/2$, we have
   \begin{align*}
    \sep(\nu;\kappa_1,\kappa_2)=0.
    \end{align*}
  \end{lem}

 \begin{lem}[{cf.~\cite[Section $3\frac{1}{2}.33$]{gromov}}]\label{inl3}Let $X$ be
  an mm-space. Then, for any $\kappa,\kappa'>0$ with $\kappa>\kappa'$, we have
  \begin{align*}
   \obs_{\mathbb{R}}(X;-\kappa') \geq \sep (X;\kappa,\kappa).
   \end{align*}
  \end{lem}
See also \cite[Lemma 2.5]{funa3} for the proof of the above lemma.

  Let $N$ be a CAT$(0)$-space and $\nu\in \mathcal{P}^2(N)$. Given any
  $\kappa>0$, we define the \emph{central radius} $\crad(\nu,1-\kappa)$
  as the infimum of $\rho >0$ such that $\nu(B_N(b(\nu),\rho))\geq
  1-\kappa$. Let $X$ be an mm-space and $N$ a CAT$(0)$-space such that $f_{\ast}(\mu_X)\in
\mathcal{P}^2(N)$ for any $1$-Lipschitz map $f:X\to N$. For any $\kappa >0$, we define 
\begin{align*}
\obsc_N(X;-\kappa):= \sup \{ \crad(f_{\ast}(\mu_X),1-\kappa) \mid f:X\to
 N \text{ is a }1 \text{-Lipschitz map}\},
\end{align*}and call it the \emph{observable central radius} of $X$.

From the definition, we immediately obtain the following lemma.
  \begin{lem}[{cf.~\cite[Section $3\frac{1}{2}.31$]{gromov}}]\label{inl4}For any
   $\kappa>0$, we have
   \begin{align*}
   \obs_{\mathbb{R}}(X;-\kappa) \leq 2\obsc_{\mathbb{R}}(X;-\kappa).
    \end{align*}
   \end{lem}Observable diameters, separation distances, observable
   central radii are introduced by Gromov in \cite[Chapter
   $3\frac{1}{2}$]{gromov} to capture the theory of the L\'{e}vy-Milman
   concentration of $1$-Lipschitz maps visually.

   Given an mm-space $X$, we define the concentration function
   $\alpha_X:(0,+\infty)\to \mathbb{R}$ of $X$ as the supremum of
   $\mu_X(X\setminus A_{+r})$, where $A$ runs over all Borel subsets of
   $X$ such that $\mu_X(A)\geq 1/2$ and $A_{+r}$ is an open
   $r$-neighborhood of $A$. Concentration functions were introduced by
   D. Amir and V. Milman in \cite{amir}.
           \section{Proof of the main theorem}

           \begin{lem}\label{pl1}Let $N$ be a CAT$(0)$-space. Then, for any $n\in
            \mathbb{N}$, the map
            \begin{align*}
             s_n: N^{\otimes n}\ni
             (x_1,x_2,\cdots,x_n)\mapsto \frac{1}{n}\wa x_i \in N
             \end{align*}is $(1/n)$-Lipschitz with respect to the
            $\ell^1$-distance function on the product space $N^{\otimes n}$.
            \begin{proof}Assuming that the map $s_{n-1}$ is
             $1/(n-1)$-Lipschitz, by Lemma \ref{cat1}, we have 
             \begin{align*}
             \dist_N(s_n((x_i)_{i=1}^n), s_n((y_i)_{i=1}^n))\leq \ &
              \Big(1-\frac{1}{n}\Big)
              \dist_N(s_{n-1}((x_i)_{i=1}^{n-1}),
              s_{n-1}((y_i)_{i=1}^{n-1}))+ \frac{1}{n}
              \dist_N(x_n,y_n)\\
              \leq \ & \Big(1-\frac{1}{n}\Big)\frac{1}{n-1}
              \sum_{i=1}^{n-1}\dist_N(x_i,y_i) +
              \frac{1}{n}\dist_N(x_n,y_n)\\
              = \ & \frac{1}{n}\sum_{i=1}^n \dist_N(x_i,y_i).
              \end{align*}This completes the proof.
             \end{proof}
            \end{lem}

            To prove Theorem \ref{mth1}, we need the following two theorems.
            \begin{thm}[{cf.~\cite[Lemma $5.5$]{funa1}}]\label{pt1}Let $\nu$ be a Borel probability measure on an
             $\mathbb{R}$-tree such that $\nu\in
             \mathcal{P}^2(T)$. Then, there exists a $1$-Lipschitz
             function $\varphi_{\nu}:T\to \mathbb{R}$ such that
             \begin{align*}
              \crad(\nu,1-\kappa)\leq
              \ &\crad((\varphi_{\nu})_{\ast}(\nu),1-\kappa)+\sep\Big(\nu;\frac{1}{3},\frac{\kappa}{2}\Big)\\
              &\hspace{2cm} +\sep\Big((\varphi_{\nu})_{\ast}(\nu);\frac{1}{3},\frac{\kappa}{2}\Big)
              + \sep((\varphi_{\nu})_{\ast}(\nu);1-\kappa,1-\kappa)
              \end{align*}for any $\kappa>0$.
             \end{thm}

            \begin{thm}[{cf.~\cite[Corollary 1.17]{ledoux}}]\label{pt2}Let $X=X_1
             \otimes  \cdots \otimes X_n$ be a product mm-space of
             mm-spaces $X_i$ with finite diameter $D_i$, $i=1,\cdots,n$, equipped with the product
             probability measure $\mu_X:=\mu_{X_1}\otimes \cdots \otimes
             \mu_{X_n}$ and the $\ell^1$-distance function
             $\dist_{\ell^1}:=\sum_{i=1}^n \dist_{X_i}$. Then, for any $1$-Lipschitz function $f:X\to \mathbb{R}$
             and any $r>0$, we have
             \begin{align}\label{ps5}
              \mu_{X}(\{ x\in X \mid |f(x)-\mathbb{E}_{\mu_X}(f)|\geq r\})\leq 2e^{-r^2/2D^2},
              \end{align}where $D^2:=\sum_{i=1}^n D_i^2$. Moreover, we have
             \begin{align}\label{ps6}
              \alpha_X(r)\leq e^{-r^2/8D^2}.
              \end{align}
             \end{thm}
           \begin{proof}[Proof of Theorem $\ref{mth1}$]Let
            $s_n:T^{\otimes n}\to T$ be a map which sends every point
            in $T^{\otimes n}$ to its inductive mean value. Putting $\nu:=(Y_1)_{\ast}\mathbb{P}$, we first prove the following.
            \begin{claim}\label{pcl1}We have
             \begin{align*}
                                                        \nu^{\otimes n}(\{x\in T^{\otimes n} \mid
             \dist_N(s_n(x),\mathbb{E}_{\nu^{\otimes n}}(s_n))\geq r\})\leq
             4e^{-\frac{nr^2}{75 D^2}}.
             \end{align*}
            \begin{proof}
                                                       Since
            the metric space $(T,n \dist_T)$ is an $\mathbb{R}$-tree,
            by virtue of Theorem \ref{pt1}, there
            exists a $1$-Lipschitz function $\varphi_n:(T,n\dist_T)\to \mathbb{R}$
            such that
            \begin{align*}
             &n\crad((s_n)_{\ast}(\nu^{\otimes n}),1-\kappa)\\ \leq \ &\crad
             ((\varphi_{n}\circ s_n)_{\ast}(\nu^{\otimes n}),1-\kappa)+n \sep
             \Big((s_n)_{\ast}(\nu^{\otimes n});\frac{1}{3},\frac{\kappa}{2}\Big)\\
             & +\sep\Big((\varphi_{n}\circ s_n)_{\ast}(\nu^{\otimes
             n});\frac{1}{3},\frac{\kappa}{2}\Big)
              + \sep((\varphi_{n}\circ s_n)_{\ast}(\nu^{\otimes n});1-\kappa,1-\kappa)
             \end{align*}for any $\kappa>0$. By Lemma \ref{pl1}, the
            function $\varphi_n
            \circ s_n:(T^{\otimes n},\dist_{\ell^1})\to \mathbb{R}$ is
            $1$-Lipschitz. Combining Lemma \ref{inl1} with Lemmas
             \ref{inl2}, \ref{inl3}, and
            \ref{inl4}, for any
            $\kappa,\kappa'>0$ such that $\kappa'<\kappa<1/2$, we hence have 
            \begin{align*}
             n\crad ((s_n)_{\ast}(\nu^{\otimes n}),1-\kappa) \leq \ &\crad
             ((\varphi_{n}\circ s_n)_{\ast}(\nu^{\otimes n}),1-\kappa)+n \sep
             \Big((s_n)_{\ast}(\nu^{\otimes n});\frac{1}{3},\frac{\kappa}{2}\Big)\\
             & \hspace{5cm}+\sep\Big((\varphi_{n}\circ s_n)_{\ast}(\nu^{\otimes
             n});\frac{1}{3},\frac{\kappa}{2}\Big)\\
             \leq \ & \obsc_{\mathbb{R}}((T^{\otimes n},\dist_{\ell^1},
             \nu^{\otimes n});-\kappa)+2\sep\Big(\nu^{\otimes
             n};\frac{\kappa}{2},\frac{\kappa}{2}\Big)\\  \leq \ & \obsc_{\mathbb{R}}((T^{\otimes n},\dist_{\ell^1},
             \nu^{\otimes n});-\kappa)\\ &\hspace{3cm}+2\obs_{\mathbb{R}}((T^{\otimes n},\dist_{\ell^1},
             \nu^{\otimes n});-\kappa'/2)\\
             \leq \ & 5\obsc_{\mathbb{R}}((T^{\otimes n},\dist_{\ell^1},
             \nu^{\otimes n});-\kappa'/2).
             \end{align*}According to the inequality (\ref{ps5}), we thus get
            \begin{align*}
             n\crad((s_n)_{\ast}(\nu^{\otimes n}),1-\kappa)\leq
             5D\sqrt{2n\log(4/\kappa')} .
             \end{align*}Letting $\kappa'\to \kappa$ yields that
            \begin{align}\label{ps1}
             \crad((s_n)_{\ast}(\nu^{\otimes n}),1-\kappa)\leq
             5D\sqrt{\frac{2}{n}\log \frac{4}{\kappa}}
             \end{align}for any $\kappa\in (0,1/2)$. Given $\kappa\geq
            1/2$, taking an arbitrary $\kappa'\in (0,1/2)$, we also estimate
            \begin{align*}
             \crad((s_n)_{\ast}(\nu^{\otimes n}),1-\kappa)\leq \ &\crad((s_n)_{\ast}(\nu^{\otimes
             n}),1-\kappa')\\
              \leq \ &5D \sqrt{\frac{2}{n}\log \frac{4}{\kappa'}}\\
             = \ & 5D \frac{\sqrt{\log \frac{4}{\kappa'}}}{\sqrt{\log
             \frac{4}{\kappa}}}\sqrt{\frac{2}{n}\log \frac{4}{\kappa}}\\
             \leq \ &5D \frac{\sqrt{\log \frac{4}{\kappa'}}}{\sqrt{\log
             4}}\sqrt{\frac{2}{n}\log \frac{4}{\kappa}}\\
             \end{align*}Letting $\kappa' \to 1/2$, we hence get
            \begin{align}\label{ps2}
             \crad ((s_n)_{\ast}(\nu^{\otimes n}),1-\kappa)\leq
             5D\sqrt{\frac{3}{n} \log \frac{4}{\kappa}}.
             \end{align}The above two inequalities (\ref{ps1}) and (\ref{ps2})
             imply the claim.
                                                       \end{proof}
            \end{claim}

            Put $a_n:=\dist_T(\mathbb{E}_{\nu^{\otimes
            n}}(s_n),b(\nu))$. By Sturm's inequality (\ref{ints1}), we have
            \begin{align*}
             \int_{T^{\otimes n}} \dist_T(s_n(x),b(\nu))^2d\nu^{\otimes
             n}(x)\leq \frac{1}{n}\int_T\dist_T(x,b(\nu))^2d\nu(x).
             \end{align*}Lemma \ref{cal2} together with Lemma \ref{cal3}
            thus implies that
            \begin{align*}
             a_n^2\leq \int_{T^{\otimes
             n}}\dist_T(s_n(x),b(\nu))^2d\nu^{\otimes n}(x)\leq
             \frac{1}{n}\int_T\dist_T(x,b(\nu))^2d\nu (x)\leq \frac{4D^2}{n}.
             \end{align*}For any $r> a_n$, by using Claim \ref{pcl1}, we therefore obtain
            \begin{align*}
             &\mathbb{P}\Big(\Big\{ \omega \in \Omega \mid \dist_T\Big(\frac{1}{n} \wa Y_i(\omega) ,
             \mathbb{E}_{\mathbb{P}}(Y_1)\Big)\geq r \Big\}\Big)\\ = \ &\nu^{\otimes
             n}(\{x\in T^{\otimes n} \mid \dist_T(s_n(x),b(\nu))\geq
             r\}) \tag*{}\\
             \leq \ & \nu^{\otimes n}( \{x\in T^{\otimes n} \mid
             \dist_T(s_n (x), \mathbb{E}_{\nu^{\otimes n}}(s_n))\geq r-
             a_n \}) \tag*{}\\
             \leq \ & 4e^{-\frac{n(r-a_n)^2}{75 D^2}} \tag*{}\\
             \leq \ & 4 e^{\frac{n a_n^2}{75D^2}}e^{-\frac{nr^2}{150
             D^2}} \tag*{}\\
             \leq \  & 4e^{\frac{4}{75}}e^{-\frac{nr^2}{150 D^2}}. \tag*{}
             \end{align*}If $r\leq a_n$, then we have
            \begin{align*}
             \mathbb{P}\Big(\Big\{ \omega \in \Omega \mid \dist_T\Big(\frac{1}{n} \wa Y_i(\omega) ,
             \mathbb{E}_{\mathbb{P}}(Y_1)\Big)\geq r \Big\}\Big)\leq 
             e^{\frac{na_n^2}{150D^2 }}e^{-\frac{na_n^2}{150 D^2}}<
             e^{\frac{2}{75}}e^{-\frac{nr^2}{150D^2}}<4e^{\frac{4}{75}}e^{-\frac{nr^2}{150
             D^2}}.
             \end{align*}Combining these two inequalities completes the
            proof of the theorem.
            \end{proof}Theorem \ref{mth2} follows from the same proof of
            Theorem \ref{mth1} together with the inequality (\ref{ps6}) and
            the following theorem. We
            shall consider an mm-space satisfying
            \begin{align}\label{ps3}
             \alpha_X(r)\leq C_Xe^{-c_Xr^2}
             \end{align}for some positive constants $c_X,C_X>0$ and any
             $r>0$. For such an mm-space $X$ and $m\in \mathbb{N}$, we put
 \begin{align*}
   A_{m,X}:= 1+\frac{\sqrt{\pi}e^{(m+1)/(4m-2)}}{2} \max\{e^{(\pi
   C_X)^2/2}, 2C_X e^{(\pi C_X)^2}\}
   \end{align*}and
   \begin{align*}
    \widetilde{A}_{m,X}:=1+ \sqrt{\pi}C_X e^{(m+1)/(4m-2)}.
    \end{align*}
            \begin{thm}[{cf.~\cite[Theorem 1.1]{funa2}}]Let an mm-space $X$
             satisfies $(\ref{ps3})$, $N$ be an $m$-dimensional Hadamard
             manifold, and $f:X\to N$ a $1$-Lipschitz map.
             Then, for any $r>0$, we have
             \begin{align*}
              \mu_X(\{x\in X \mid \dist_N(f(x),\mathbb{E}_{\mu_X}(f))\geq
              r\})\leq \min\{A_{m,X} e^{-(c_X/(8m))r^2},\widetilde{A}_{m,X} e^{-(c_X/(16m))r^2} \}.
              \end{align*}
             \end{thm}

      %       \begin{rem}\upshape 
      %        \end{rem}

          %  \begin{proof}[Proof of Theorem \ref{}]
          %   \begin{align*}
          %    \nu^{\otimes n}(\{ x\in N \mid
          %    \dist_N(s_n(x),\mathbb{E}(s_n))\geq r)\leq \ & \min \{B_m e^{-\frac{}{8nD^2}} \}\\
           %   \end{align*}
           %  \begin{align*}
           %     \mathbb{P}\Big(\Big\{ \omega \in \Omega \mid \dist_N\Big(\frac{1}{n} \wa Y_i(\omega) ,
           %  \mathbb{E}(Y_1)\Big)\geq r \Big\}\Big) = \ &\nu^{\otimes
           %  n}(\{x\in N^{\otimes n} \mid \dist_N(s_n(x),b(\nu))\geq
           %   r\})\\ \leq \ & \nu^{\otimes n}( \{x\in N^{\otimes n} \mid
           %  \dist_N(s_n (x), \mathbb{E}(s_n))\geq r- a_n \})\\
           %   \end{align*}
           %  \end{proof}

	\end{document}